\documentclass[11pt]{article}
\usepackage{amssymb}
\usepackage{url}
\usepackage{color}
\usepackage{amsfonts}
\usepackage{mathrsfs}
\usepackage{appendix}
\usepackage{amsfonts,amsmath,amsthm,amssymb,cases,mathrsfs}
\DeclareMathOperator{\esssup}{esssup}

\newtheorem{lemma}{Lemma}[section]
\newtheorem{definition}[lemma]{Definition}
\newtheorem{theorem}[lemma]{Theorem}
\newtheorem{prop}[lemma]{Proposition}
\newtheorem{rmk}[lemma]{Remark}

\numberwithin{equation}{section} 
\title{\huge Dynamic Programming Principle and Associated Hamilton-Jacobi-Bellman Equation for
Stochastic Recursive Control Problem with Non-Lipschitz Aggregator\thanks{JP is supported by NSF of China (No.11426151); QZ is supported by NSF of China (No. 11101090, 11471079) and the Science and Technology Commission of Shanghai Municipality (No. 14XD1400400).}}
\author{ \Large Jiangyan Pu\footnote{School of International Finance, Shanghai Finance University, Shanghai 201209, China. (Email: \texttt{pujy@shfc.edu.cn})
}
~~~and~~~\Large Qi Zhang\footnote{School of Mathematical Sciences, Fudan
University, Shanghai 200433, China. (Email:
\texttt{qzh@fudan.edu.cn})}
}

\date{}
\allowdisplaybreaks
\begin{document}
\maketitle

 \textbf{Abstract.} In this work we study the stochastic recursive control problem, in which the aggregator (or called generator) of the backward stochastic differential equation describing the running cost is continuous but not necessarily Lipschitz with respect to the first unknown variable and the control, and monotonic with respect to the first unknown variable. The dynamic programming principle and the connection between the value function and the viscosity solution of the associated Hamilton-Jacobi-Bellman equation are established in this setting by the generalized comparison theorem of backward stochastic differential equations and the stability of viscosity solutions. Finally we take the control problem of continuous-time Epstein-Zin utility with non-Lipschitz aggregator as an example to demonstrate the application of our study.

\textbf{Key words:} stochastic recursive control problem, non-Lipschitz aggregator, dynamic programming principle, Hamilton-Jacobi-Bellman equation,  continuous-time Epstein-Zin utility, viscosity solution.

\textbf{Mathematics Subject Classification:}  93E20, 90C39, 35K10

\large
\section{Introduction}

The stochastic control theory arose along with the birth of stochastic analysis and developed fast in the last few decades due to its wide applications. Indeed the stochastic control system is a natural and effective way to involve the uncertainty, disturbance and ambiguity appearing in the real-world control problems. Its powerful feature is especially embodied in the mathematical finance problems as we study the pricing of contingent claim and the optimal strategy in the stochastic financial models, which on the contrary promotes the development of stochastic control theory.

In the development of stochastic control theory, the backward stochastic differential equation (BSDE for short) plays a big role. First of all, linear BSDE itself originated from the study of maximum principle for a stochastic control problem in Bismut \cite{bis1} (1973), where it appears as the adjoint equation, and later the application of this pioneer work to mathematical finance was presented by Bismut \cite{bis2} (1975). The maximum principle reveals that the optimal solution of a stochastic control problem can be depicted by the stochastic Hamiltonian system which is actually a forward-backward stochastic differential equation (FBSDE for short). Furthermore, when the stochastic control system is observed partially or the state equation itself is a stochastic partial differential equation, the adjoint equation in this case is a backward stochastic partial differential equation, which was indicated in Bensoussan \cite{ben} (1983). The maximum principle for stochastic control system with the diffusion term dependent on control and the control regions not necessarily convex was another mile of stone in stochastic control theory, which was solved in Peng \cite{pen1} (1990) by using the second-order matrix-valued BSDE to serve as the adjoint equation. We recommend the reader to refer to the monograph \cite{yo-zh} (1999) by Yong and Zhou, in which comprehensive introductions to the stochastic control theories are presented.

The role of BSDE in stochastic control theory is not only restricted to the maximum principle as an adjoint equation. The nonlinear BSDE has even greater influences in the stochastic recursive utilities and their associated control problems, thanks to the importance of recursive utilities in modern mathematical finance. The existence and uniqueness of adapted solution to nonlinear BSDE in mathematics was proved by Pardoux and Peng \cite{par-pen1} (1990). Later Duffie and Epstein \cite{Duffie} (1992) put forward the concept of stochastic differential utility in a conditional expectation form which is equivalent to the nonlinear BSDE. Since then both BSDEs and stochastic control problems in mathematical finance achieved a great progress benefiting from their connections. The reader can refer to El Karoui, Peng and Quenez \cite{el} (1997) which concluded early works on BSDEs and their applications to mathematical finance. 

The stochastic recursive control problem we concern with was introduced by Peng \cite{Peng1992} (1992). Its state equation is a stochastic differential equation (SDE for short):
\begin{equation}\label{eq1a}
\begin{aligned}
X_s^{t,x;v}=x+\int_t^sb(r,X_r^{t,x;v},v_r)dr +\int_t^s\sigma (r,X_r^{t,x;v},v_r)dB_r\ \ \ \ {\rm for}\ x\in\mathbb{R}^n.
\end{aligned}
\end{equation}
The cost functional is associated with the solution to a BSDE on the interval $[t,T]$ coupled with the state process:
\begin{equation}\label{eq5a}
Y_s^{t,x;v}=h(X_T^{t,x;v})+\int_s^Tf(r,X_r^{t,x;v},Y_r^{t,x;v},Z_r^{t,x;v},v_r)dr-\int_s^TZ_r^{t,x;v} dB_r
\end{equation}
and defined as below
\begin{equation}\label{eq9}
J(t,x;v)\triangleq  Y_t^{t,x;v}, 
\end{equation}
where $v$ is an admissible control process in the admissible control set $\mathcal{U}$. 
The corresponding control problem is to find an optimal $\bar{v}\in\mathcal{U}$ to maximize the cost functional (\ref{eq9}) for given $(t,x)$. As you can see, FBSDE arises again to depict this recursive control system. Actually, as a popular equation, FBSDEs appear in numerous control and related mathematical finance problems. For the theories and applications of FBSDEs, we recommend the reader to refer to e.g. Ma, Protter and Yong \cite{ma-pr-yo} (1994), Peng and Wu \cite{Pengwu} (1999), Yong \cite{yo} (2010) or the classical book \cite{ma-yo} (1999) by Ma and Yong.

For this stochastic recursive control system (\ref{eq1a})--(\ref{eq9}), Peng \cite{Peng1992} established the dynamic programming principle in the Lipschitz setting of the aggregator (or called generator from BSDE point of view) and connected its value function with the Hamilton-Jacobi-Bellman (HJB for short) equation.
Since the recursive utility can be regarded as the solution of BSDE (\ref{eq5a}) with the conditional expectation form, the stochastic recursive control system in form includes the control problem related to stochastic (recursive) differential utilities
\begin{equation}\label{pz36}
V_t=E^{\mathscr{F}_t}[\int_t^Tf(c_s,V_s)ds],
\end{equation}
where $c$ is the consumption process serving as the control. It is well known that the recursive utility is an extension of the time-additive expected utility. In comparison with the latter, the former's risk aversion and intertemporal substitutability are separated in the aggregator
which is ``useful in
clarifying the determinants of asset prices and presumably for a number of other issues in capital theory
and finance" (see \cite{Duffie}).

In our study we aim to relax the Lipschitz restriction to the aggregator, i.e. the aggregator $f(c,u)$ in (\ref{pz36}) is continuous but not necessarily Lipschitz with respect to both the utility variable $u$ and the consumption variable $c$, moreover, it is of polynomial growth with respect to $u$ in our assumptions. These settings would make much difference in the deduction of the dynamic programming principle and bring much trouble in the verification of conditions for the stability of viscosity solutions which leads to the connection between the value function and the viscosity solution of the associated HJB equation. Although there are some further results on stochastic recursive control problem from the dynamic programming principle point of view, such as Peng \cite{Peng} (1997) for non-Markovian framework, Buckdahn and Li \cite{buc-li} (2008) for stochastic differential games, Wu and Yu \cite{Wu} (2008)
for the cost functional generated by reflected BSDE, Li and Peng \cite{Li} (2009) for the cost functional generated by BSDE with jumps, Chen and Wu \cite{che-wu} (2012) for the state equation with delay, etc.,
as far as we know there are no existing results in the non-Lipschitz aggregator setting.
However, back to the stochastic recursive utilities, the aggregators in many situations are not Lipschitz with respect to the utilities and consumptions. For instance, the aggregator of the well-known continuous-time Epstein-Zin utility has a form
\begin{equation}\label{eq6}
f(c,u)=\frac{\delta}{1-{1\over\psi}}(1-\gamma)u\Big[ (\frac{c}{((1-\gamma)u)^{\frac{1}{1-\gamma}}})^{1-\frac{1}{\psi}} - 1 \Big],
\end{equation}
where $\delta >0$ is the rate of time preference, $0 < \gamma \neq 1$ is the coefficient of relative risk aversion and $0 < \psi \neq 1$ is the elasticity of intertemporal substitution. 
In general, the aggregator $f(c,u)$ in (\ref{eq6}) is not Lipschitz with respect to $c$ and $u$ but could be monotonic with respect to the latter by suitable choices of parameters. We notice that a remarkable process for stochastic recursive control problem with non-Lipschitz aggregator had been made by Kraft, Seifried and Steffensen \cite{Kraft}, in which the verification theorem is proved for the non-Lipschitz Epstein-Zin aggregator
and explicit solutions to HJB equation are given in some cases. Nevertheless, the dynamic programming principle for non-Lipschitz stochastic recursive control system is still not involved. 

Certainly, one important technique to study stochastic recursive control problem in the non-Lipschitz setting is how to deal with the BSDE with non-Lipschitz aggregator. There is much literature devoting to the relaxation of Lipschitz condition of the aggregator $f(t,y,z)$ with respect to the first unknown variable $y$ and/or the second unknown variable $z$, such as Lepeltier and San Martin {\cite{LS2} (1997) for linear growth condition of $y$ and $z$, Kobylanski \cite{K1} (2000) for quadratic growth condition of $z$, Briand and Carmona \cite{BC} (2000) for polynomial growth condition of $y$ and Pardoux \cite{Pardoux} (1999) for arbitrary growth condition of $y$, to name but a few. As for the monotonic condition of $y$ it was first introduced to BSDE theory by Peng \cite{pen2} (1991) for the infinite horizon BSDE. After that many works adopted the monotonic condition to weaken the Lipschitz assumption or make BSDE more applicable to the related fields, including e.g. Hu and Peng \cite{Hu} (1995), 
Pardoux and Tang \cite{Pardoux3} (1999), Briand, Delyon, Hu, Pardoux and Stoica \cite{Briand03} (2003), besides \cite{BC}, \cite{Pardoux} and \cite{Pengwu} mentioned above. However, to our best knowledge there are no existing results about the dynamic programming principle and associated HJB equation for a stochastic control system involving nonlinear BSDE with the monotonic or other non-Lipschitz aggregators.

This paper generalizes the results in \cite{Peng1992} by studying a stochastic recursive control problem where the cost functional generated by BSDE with the non-Lipschitz but continuous and monotonic aggregator. We first establish the dynamic programming principle with the helps of the backward semigroups and generalized comparison theorem in non-Lipschitz setting, and then connect the value function of our concerned control problem with a viscosity solution of the corresponding HJB equation by means of stability of viscosity solution. Needless to say, the relaxation of Lipschitz condition makes our control problem applicable to more mathematical finance models, including the continuous-time Epstein-Zin utility with non-Lipschitz aggregator.

The rest of this paper is organized as follows. In Section 2, some useful notation is introduced and the necessary preliminaries are clarified. Then we deduce the dynamic programming principle in a non-Lipschitz aggregator setting in Section 3. In Section 4 we establish the relationship between the value function of the control problem and the viscosity solution of the corresponding HJB equation provided that the aggregator of BSDE independent of the second unknown variable. Finally, an example from the control problem of continuous-time Epstein-Zin utilities is given in Section 5 to demonstrate the application of our work to mathematical finance. 

\section{Notation and preliminaries}

Given a complete probability space $(\Omega, \mathscr{F}, P)$, let $(B_s)_{0 \leq s \leq T}$ be a $d$-dimensional Brownian motion on the probability space. Denote by $(\mathscr{F}_s)_{0 \leq s \leq T}$ the nature filtration generated by $(B_s)_{0 \leq s \leq T}$ with $\mathscr{F}_0$ containing all $P$-null sets of $\mathscr{F}$.
We use $|\cdot|$ and $\langle \cdot, \cdot \rangle$ throughout the paper to denote the Euclidean norm and dotproduct, respectively,
and then we define some useful notation.
\begin{definition}
For $q\geq1$, $0\leq t\leq T$, we denote by

$\bullet$ $L^{2q}(\Omega,\mathscr{F}_t;\mathbb{R}^n)$: the space of all $\mathscr{F}_t$-measurable random variables $\xi:\Omega\to\mathbb{R}^n$ satisfying $E[|\xi|^{2q}]<\infty$;

$\bullet$ $S^{2q}(t,T;\mathbb{R}^n)$: the space of all joint measurable processes $\varphi:[t,T]\times\Omega\to\mathbb{R}^n$ satisfying
\begin{enumerate}
\item [(i)] $\varphi_s$ is $\mathscr{F}_{s}$-adapted measurable and $\varphi_s$ is a.s. continuous for $t\leq s\leq T$,
\item [(ii)] $E[\sup\limits_{s\in [t,T]} |\varphi(s)|^{2q}] < \infty$;
\end{enumerate}

$\bullet$ $M^{2q}(t,T;\mathbb{R}^n)$: the space of all joint measurable processes $\varphi:[t,T]\times\Omega\to\mathbb{R}^n$ satisfying
\begin{enumerate}
\item [(i)] $\varphi_s$ is $\mathscr{F}_{s}$-adapted measurable for $t\leq s\leq T$,
\item [(ii)] $E[\int_t^T|\varphi(s)|^{2q}ds] < \infty$.
\end{enumerate}
\end{definition}

Next we clarify the set of admissible control processes $\mathcal{U}$ in the control system (\ref{eq1a})--(\ref{eq9}) which is defined as below:
\begin{equation*}
\mathcal{U}\triangleq \{v|\ v\in M^2(0,T;\mathbb{R}^m)\ \text{and takes values in a compact set}\ U\subset\mathbb{R}^m \}.
\end{equation*}



We assume the conditions to the coefficients of state equation (\ref{eq1a}) as follows.\\
(H1) Both $b(t,x,v): [0,T]\times \mathbb{R}^n \times U \rightarrow \mathbb{R}^n$ and $\sigma(t,x,v): [0,T]\times \mathbb{R}^n \times U \rightarrow \mathbb{R}^{n\times d}$ are joint measurable and continuous with respect to $t$.\\
(H2) For any $t \in [0,T]$, $x, x^\prime \in \mathbb{R}^n$, $v, v^\prime \in U$, there exists a constant $L \geq 0$ such that
\begin{eqnarray*}
|b(t,x,v)-b(t,x^\prime,v^\prime)|+|\sigma(t,x,v)-\sigma(t,x^\prime,v^\prime)|\leq L(|x-x^\prime|+|v-v^\prime|). \label{eq3}
\end{eqnarray*}

A standard argument for SDE with Lipschitz condition leads to the existence and uniqueness result to the solution of SDE (\ref{eq1a}). For the use of the proof for dynamic programming principle, we consider a general SDE with a random variable initial value, and conclude the existence result, uniqueness result and some useful estimates to its solution.
\begin{prop}\label{pz24}
Assume Conditions (H1)--(H2). Given $q\geq1$, for any $t\in[0,T]$, $v\in \mathcal{U}$, $\eta\in L^{2q}(\Omega,\mathscr{F}_t;\mathbb{R}^n)$, the following SDE
\begin{equation}\label{pz32}
\begin{aligned}
X_s^{t,\eta;v}=\eta+\int_t^sb(r,X_r^{t,\eta;v},v_r)dr +\int_t^s\sigma (r,X_r^{t,\eta;v},v_r)dB_r
\end{aligned}
\end{equation}
has a unique strong solution $X_\cdot^{t,\eta;v}\in S^{2q}(t,T;\mathbb{R}^n)$.

Moreover, there exists a constant $C>0$ depending only on $L, T$ such that for any $t\leq s\leq T$, $v, v^\prime\in \mathcal{U}$, $\eta, \eta^\prime \in L^{2q}(\Omega,\mathscr{F}_t;\mathbb{R}^n)$, we have
\begin{equation*}\label{eq5}
E\Big[\sup_{s\in [t,T]}|X_s^{t,\eta;v}|^{2q}\Big] \leq C(1+E[|\eta|^{2q}+\int_t^T|v_r|^{2q}dr])
\end{equation*}
and
\begin{equation*}\label{eq4}
E^{\mathscr{F}_t}\Big[\sup_{s\in [t,T]}|X_s^{t,\eta;v}- X_s^{t,\eta^\prime ;v^\prime}|^{2q}\Big] \leq C(|\eta-\eta^\prime|^{2q} +E^{\mathscr{F}_t}[\int_t^T|v_r-v_r^\prime|^{2q}dr]).
\end{equation*}
\end{prop}


Then we turn to the assumptions to the coefficients of BSDE (\ref{eq5a}).\\
(H3) Both $h(x): \mathbb{R}^n\to\mathbb{R}^n$ and $f(t,x,y,z,v): [0,T]\times \mathbb{R}^n\times\mathbb{R}^1\times\mathbb{R}^d\times U \to\mathbb{R}^n$ are joint  measurable, and $f(t,x,y,z,v)$ is continuous with respect to $(t,y,v)$.
\\
(H4) For any $t \in [0,T]$, $x, x^\prime \in \mathbb{R}^n$, $y\in \mathbb{R}^1$, $z,z^\prime \in \mathbb{R}^d$, $v\in U$, there exists a constant $\lambda \geq 0$ such that
\begin{equation*}\label{eq7}
\begin{aligned}
|h(x)-h(x^\prime)|+|f(t,x,y,z,v)-f(t,x^\prime,y, z^\prime,v)|\leq \lambda ( |x-x^\prime|+|z-z^\prime|).
\end{aligned}
\end{equation*}
(H5) For any $t \in [0,T]$, $x\in \mathbb{R}^n$, $y,y^\prime \in \mathbb{R}^1$, $z\in \mathbb{R}^d$, $v\in U$, there exists a constant $\mu \in \mathbb{R}^1$ such that
\begin{equation*}\label{eq8}
(y-y^\prime)\big(f(t,x,y,z,v)-f(t,x,y^\prime,z,v)\big)\leq \mu |y-y^\prime|^2.
\end{equation*}
(H6) For a given $p \ge 1$ and any $t \in [0,T]$, $x\in \mathbb{R}^n$, $y\in \mathbb{R}^1$, $z\in \mathbb{R}^d$, $v\in U$, there exists a constant $\kappa>0$
such that $$|f(t,x,y,z,v)-f(t,x,0,z,v)| \leq \kappa(1+ |y|^p).$$

With Conditions (H1)--(H6), the existence and uniqueness of solution of BSDE  \eqref{eq5a} is an existing result and we recommend the reader to refer to \cite{Pardoux} for details. Also here we consider a general BSDE coupled with the solution of SDE (\ref{pz32}), and conclude the existence result, uniqueness result and some useful estimates to its solution.
\begin{prop}\label{pz27}
Assume Conditions (H1)--(H6). Given $q\geq1$, for any $t\in[0,T]$, $v\in \mathcal{U}$, $\eta\in L^{2q}(\Omega,\mathscr{F}_t;\mathbb{R}^n)$, the following BSDE
\begin{equation}\label{pz25}
\begin{aligned}
Y_s^{t,\eta;v}=h(X_T^{t,\eta;v})+\int_s^Tf(r,X_r^{t,\eta;v},Y_r^{t,\eta;v},Z_r^{t,\eta;v},v_r)dr-\int_s^TZ_r^{t,\eta;v} dB_r
\end{aligned}
\end{equation}
has a unique solution $(Y_\cdot^{t,\eta;v},Z_\cdot^{t,\eta;v})\in S^{2q}(t,T;\mathbb{R}^1) \times M^2(t,T;\mathbb{R}^d)$.

Moreover, there exists a constant $C>0$ depending only on $L, \lambda, \mu, \kappa, T$ such that for any $t\leq s\leq T$, $v\in \mathcal{U}$, $\eta, \eta^\prime \in L^{2q}(\Omega,\mathscr{F}_t;\mathbb{R}^n)$, we have
\begin{equation*}
|Y_t^{t,\eta;v}|^{2q}\leq C\big(1+|\eta|^{2q}+E^{\mathscr{F}_t}[\int_t^T|f(r,0,0,0,v_r)|^{2q}dr]),
\end{equation*}
\begin{eqnarray*}
|Y_t^{t,\eta;v}- Y_t^{t,\eta^\prime ;v}|
\leq C|\eta-\eta^\prime|
\end{eqnarray*}
and
\begin{equation*}
E[\sup_{s\in [t,T]}|Y_s^{t,\eta;v}|^{2q}+\int_t^T|Y_s^{t,\eta;v}|^{2q-2}|Z_s^{t,\eta;v}|^{2}]\leq C\big(1+E[|\eta|^{2q}+\int_t^T|f(r,0,0,0,v_r)|^{2q}dr]\big).
\end{equation*}
\end{prop}

\begin{rmk}
In the proof for Proposition \ref{pz27} we substitute the monotonic condition (H5) for the global Lipschitz condition in the standard deduction by It\^{o}'s formula to obtain the same forms of estimates. As for the $L^{2q}$ estimates of solutions, $q\geq1$, the common localization method is applied in the proof and the reader can refer to e.g. Lemma 3.3 in \cite{zh-zh3}.
\end{rmk}

Just as the classical situation, the comparison theorem of BSDE (\ref{eq5a}) is necessary to establish the dynamic programming principle, without the exception to the new situation that the aggregator satisfies the continuous and monotonic condition rather than the Lipschitz condition. For this setting, the following comparison theorem in Fan and Jiang \cite{Fan} is applicable.
\begin{theorem}\label{pz2} (Comparison theorem in \cite{Fan})
Let $(\xi,f)$ and $(\xi^\prime,f^\prime)$ be two generators for finite horizon BSDEs on $[0,T]$ with the corresponding solutions $(y, z)$ and $(y^\prime, z^\prime)$ in the space $S^{2}(0,T;\mathbb{R}^1) \times M^2(0,T;\mathbb{R}^d)$, respectively. Assume that $\xi, \xi^\prime\in L^{2}(\Omega,\mathscr{F}_T;\mathbb{R}^1)$ satisfy $\xi\leq\xi^\prime$ a.s. and  $f$ (resp. $f^\prime$) satisfies the following conditions:

(A1) $f(t,y,z)$ is weakly monotonic with respect to $y$, i.e. there exists a nondecreasing concave function $\rho:\mathbb{R}^+\to\mathbb{R}^+$ with $\rho(0)=0$ and $\rho(u)>0$ for $u>0$ such that $\int_{0^+}\frac{1}{\rho(u)}du= \infty$ and for any $t\in[0,T]$, $y_1, y_2\in\mathbb{R}^1$, $z\in\mathbb{R}^d$,
\begin{equation*}
\text{sgn}(y_1-y_2)\cdot(f(t,y_1,z)-f(t,y_2,z)) \leq \rho(|y_1-y_2|)\ \ \ \ {\rm a.s.};
\end{equation*}

(A2) there exists a continuous, nondecreasing and linear growth function $\phi:\mathbb{R}^+\to\mathbb{R}^+$ satisfying $\phi(0)=0$ such that for any $t\in[0,T]$, $y\in\mathbb{R}^1$, $z_1,z_2\in\mathbb{R}^d$,
\begin{equation*}
|f(t, y, z_1)-f(t, y, z_2)| \leq \phi(|z_1 - z_2|)\ \ \ \ {\rm a.s.};
\end{equation*}

(A3) for any $t\in[0,T]$, $f(t,y^\prime_t, z^\prime_t) \leq f^\prime (t,y^\prime_t, z^\prime_t)$ (resp. $f(t,y_t, z_t) \leq f^\prime (t,y_t, z_t)$).

Then we have
\begin{equation*}
y_t \leq y^\prime_t\ \ \ \ \text{for all $t \in [0,T]$\ \ \ \ {\rm a.s.}}
\end{equation*}
\end{theorem}
%
%

\section{Dynamic programming principle with non-Lipschitz aggregator}
In this section, we prove a generalized dynamic programming for stochastic recursive control problem, in which the aggregator $f$ is not necessarily Lipschitz but continuous and monotonic. To begin with, we introduce the so-called backward semigroup brought forward by Peng in \cite{Peng}.

For given $t\in[0,T]$, $t_1 \in (t,T]$, $x\in\mathbb{R}^n$, $v\in \mathcal{U}$ and $\mathscr{F}_{t_1}$-measurable $\eta \in L^2(\Omega;\mathbb{R}^1)$, we define
\begin{equation*}\label{eq13}
G^{t,x;v}_{r,t_1}[\eta]\triangleq  \hat{Y}^{t,x;v}_r, ~~r \in [t,t_1],
\end{equation*}
where $(\hat{Y}_\cdot^{t,x;v},\hat{Z}_\cdot^{t,x;v})\in S^2(t,t_1;\mathbb{R}^1) \times M^2(t,t_1;\mathbb{R}^d)$ is the solution of BSDE on the interval $[t,t_1]$:
\begin{eqnarray*}
\hat{Y}^{t,x;v}_s=\eta+\int_s^{t_1}f(r, X_r^{t,x;v},\hat{Y}^{t,x;v}_r, \hat{Z}^{t,x;v}_r,v_r)dr-\int_s^{t_1}\hat{Z}^{t,x;v}_r dB_r
\end{eqnarray*}
and $X_\cdot^{t,x;v}$ is the solution of SDE \eqref{eq1a}.

In view of the uniqueness of solution of BSDE \eqref{eq5a}, it yields that
\begin{equation*}
G_{t,T}^{t,x;v}[h(X_T^{t,x;v})]=G_{t,t+\delta}^{t,x;v}[Y_{t+\delta}^{t,x;v}].
\end{equation*}
On the other hand, back to the control system (\ref{eq1a})--(\ref{eq9}) the relevant value function of the control problem maximizing the cost functional is defined as below:
\begin{equation}\label{eq10}
u(t,x)\triangleq\esssup\limits_{v\in \mathcal{U}}J(t,x;v), \quad (t,x) \in [0,T]\times \mathbb{R}^n.
\end{equation}
In fact, $u$ is still deterministic in our non-Lipschitz setting.
\begin{lemma}\label{pz1}
Assume Conditions (H1)--(H6).  Then the cost functional $u$ defined in (\ref{eq10}) is a deterministic function.
\end{lemma}
\begin{proof}
The idea to prove this lemma was initialed by Peng \cite{Peng}.
But since the absence of Lipschitz condition to $f(t,x,y,z,v)$ with respect to $y$ and $v$, some changes should be made in the proof.

To begin with, we denote by $(\mathscr{F}_{t,s})_{t \leq s \leq T}$ the nature filtration generated by $(B_s-B_t)_{0 \leq s \leq T}$ and define two subspaces of $\mathcal{U}$:
\begin{eqnarray*}
&&\mathcal{U}^t\triangleq \{v\in\mathcal{U}|\ v_s\ \text{is}\ \mathscr{F}_{t,s}-\text{measurable}\ \text{for}\ t\leq s\leq T\};\nonumber\\
&&\bar{\mathcal{U}}^t\triangleq \{v\in\mathcal{U}|\ v_s=\sum_{j=1}^{N}v^j_sI_{A_j},\ \text{where}\ v^j\in\mathcal{U}^t\ \text{and}\ \{A_j\}_{j=1}^N\ \text{is a partition of $(\Omega,\mathscr{F}_{t})$}\}.
\end{eqnarray*}

We first prove
\begin{equation}\label{pz15}
\esssup_{v\in \mathcal{U}}J(t,x;v)=\esssup_{v\in\bar{\mathcal{U}}^t}J(t,x;v).
\end{equation}

Noticing $\bar{\mathcal{U}}^t$ is a subspace of $\mathcal{U}$, we only need to prove that ``$\leq$" holds.
To see this, note that $\bar{\mathcal{U}}^t$ is dense in $\mathcal{U}$, so for any $v\in\mathcal{U}$, there exists a sequence $\{v^n\}_{n=1}^\infty\subset\bar{\mathcal{U}}^t$ such that
\begin{eqnarray*}
\lim_{n\to\infty}E[\int_t^T|v^n_s-v_s|^2ds]=0.
\end{eqnarray*}
Moreover, we can choose a subsequence from $\{v^n\}_{n=1}^\infty$, still denoted by $\{v^n\}_{n=1}^\infty$ without loss of any generality, which satisfies $$\lim\limits_{n\to\infty}v^n_s=v_s\ \ \ \ {\rm a.s.}$$
Applying It\^{o} formula to $|Y_s^{t,x;v}-Y_s^{t,x;v^n}|^2$, together with the monotonic condition and Gronwall's inequality, we have
\begin{eqnarray}\label{pz33}
&&E[|Y_t^{t,x;v^n}-Y_t^{t,x;v}|^2]\nonumber\\
&\leq&C_pE[\int_t^T|X_s^{t,x;v^n}-X_s^{t,x;v}|^2dr]\\
&&+C_pE[\int_t^T|f(s,X_s^{t,x;v},Y_s^{t,x;v},Z_s^{t,x;v},v^n_s)-f(s,X_s^{t,x;v},Y_s^{t,x;v},Z_s^{t,x;v},v_s)|^2].\nonumber
\end{eqnarray}
Here and in the rest of this paper $C_p$ is a generic constant depending only on given parameters and its values may change from line to line, moreover, we use a bracket immediately after $C_p$ to indicate the parameters it depends on when necessary.

By (\ref{pz33}) and Propositions \ref{pz24}, it turns out that
\begin{eqnarray}\label{pz13}
&&E[|Y_t^{t,x;v^n}-Y_t^{t,x;v}|^2]\nonumber\\
&\leq&C_pE[\int_t^T|v^n_s-v_s|^2ds]\\
&&+C_pE[\int_t^T|f(s,X_s^{t,x;v},Y_s^{t,x;v},Z_s^{t,x;v},v^n_s)-f(s,X_s^{t,x;v},Y_s^{t,x;v},Z_s^{t,x;v},v_s)|^2ds].\nonumber
\end{eqnarray}
Noticing Conditions (H3), (H4) and (H6), we know that
\begin{eqnarray*}
&&|f(s,X_s^{t,x;v},Y_s^{t,x;v},Z_s^{t,x;v},v^n_s)-f(s,X_s^{t,x;v},Y_s^{t,x;v},Z_s^{t,x;v},v_s)|^2\nonumber\\
&\leq&C_p(1+|X_s^{t,x;v}|^2+|Y_s^{t,x;v}|^{2p}+|Z_s^{t,x;v}|^2),
\end{eqnarray*}
which is integrable in $L^2(\Omega\times[t,T];\mathbb{R}^1)$ in view of Propositions \ref{pz24} and \ref{pz27}. Thus by the dominated control theorem it yields
\begin{eqnarray*}
\lim_{n\to\infty}E[\int_t^T|f(s,X_s^{t,x;v},Y_s^{t,x;v},Z_s^{t,x;v},v^n_s)-f(s,X_s^{t,x;v},Y_s^{t,x;v},Z_s^{t,x;v},v_s)|^2ds]=0.
\end{eqnarray*}
Hence, taking the limits on both sides of (\ref{pz13}) we have
\begin{eqnarray*}
\lim_{n\to\infty}E[|Y_t^{t,x;v^n}-Y_t^{t,x;v}|^2]=0.
\end{eqnarray*}
Consequently, there exists a subsequence of $\{v^n\}_{n=1}^\infty$, still denoted by $\{v^n\}_{n=1}^\infty$ without loss of any generality, such that
\begin{eqnarray*}
\lim_{n\to\infty}Y_t^{t,x;v^n}=Y_t^{t,x;v}\ \ \ \ {\rm a.s.}
\end{eqnarray*}
Due to the definition of cost functionals and the arbitrariness of $v\in\mathcal{U}$, we have
\begin{equation*}
\esssup_{v\in \mathcal{U}}J(t,x;v)\leq\esssup_{v\in\bar{\mathcal{U}}^t}J(t,x;v),
\end{equation*}
and then (\ref{pz15}) follows.

The next step is to prove
\begin{equation}\label{pz16}
\esssup_{v\in\bar{\mathcal{U}}^t}J(t,x;v)=\esssup_{v\in\mathcal{U}^t}J(t,x;v),
\end{equation}
for which the argument is similar to the classical case that $f(t,x,y,z,v)$ satisfies the Lipschitz condition with respect to $y$ and $v$. The reader can refer to \cite{Peng} or \cite{Wu} for details.

Therefore, (\ref{pz15}) and (\ref{pz16}) conclude
\begin{equation*}
\esssup_{v\in \mathcal{U}}J(t,x;v)=\esssup_{v\in\mathcal{U}^t}J(t,x;v),
\end{equation*}
which implies that $u$ defined in (\ref{eq10}) is a deterministic function.
\end{proof}

With Proposition \ref{pz27} we can also obtain two lemmas related to the value function. In fact, their proofs are very similar to the counterparts in \cite{Peng}, in which the estimates in Proposition \ref{pz27} are used but the Lipschitz conditions for $f(t,x,y,z,v)$ with respect to $y$ and $v$ are not needed any more. So we leave out the proofs here.

The first lemma claims the Lipschitz continuity and linear growth of the value function $u(t,x)$ with respect to $x$. 
\begin{lemma}\label{pz26}
For any $t\in[0,T]$, $x,x^\prime\in\mathbb{R}^n$, there exists a constant $C$ such that
\begin{enumerate}
\item [(i)] $|u(t,x)-u(t,x^\prime)|\leq C|x-x^\prime|$,
\item [(ii)] $|u(t,x)|\leq C(1+|x|)$.
\end{enumerate}
\end{lemma}

The other lemma connects the cost functional with the solution of BSDE (\ref{pz25}), where the state is a random variable.
\begin{lemma}\label{pz28}
For any $t\in[0,T]$, $v\in\mathcal{U}$, $\eta\in L^{2}(\Omega,\mathscr{F}_t;\mathbb{R}^n)$, we have
\begin{equation*}
J(t,\eta;v)=Y_t^{t,\eta;v}.
\end{equation*}
\end{lemma}

Moreover, we need the following lemma  which plays a big role in the proof of dynamic programming principle.
\begin{lemma}\label{pz29}
For any $t\in[0,T]$, $v\in\mathcal{U}$, $\eta\in L^{2}(\Omega,\mathscr{F}_t;\mathbb{R}^n)$, we have
\begin{equation}\label{pz30}
u(t,\eta)\geq Y_t^{t,\eta;v}\ \ \ \ {\rm a.s.}
\end{equation}
On the other hand, for any $\varepsilon>0$, there exists an admissible control $v\in\mathcal{U}$ such that
\begin{equation}\label{pz31}
u(t,\eta)\leq Y_t^{t,\eta;v}+\varepsilon\ \ \ \ {\rm a.s.}
\end{equation}
\end{lemma}
\begin{proof}
We first prove that Lemma \ref{pz29} holds for any simple random state variable $\zeta=\sum\limits_{j=1}^{N}x_iI_{A_j}$, where $N\in\mathbb{N}$, $x_j\in\mathbb{R}^n$ and $\{A_j\}_{j=1}^N$ is a partition of $(\Omega,\mathscr{F}_{t})$.

For any $v\in\mathcal{U}$, since
\begin{equation*}
Y_t^{t,\zeta;v}=\sum_{j=1}^NY_t^{t,x_j;v}I_{A_j}\leq\sum_{j=1}^Nu(t,x_j)I_{A_j}=u(t,\zeta),
\end{equation*}
(\ref{pz30}) is true for the simple random variables. To prove (\ref{pz31}), we notice that for each $x_j$, there exists an admissible control $v_j\in\mathcal{U}^t$ such that
$$u(t,x_j)\leq Y_t^{t,x_j;v_j}+\varepsilon.$$
Hence taking $v=\sum\limits_{j=1}^Nv_jI_{A_j}\in\mathcal{U}$ we have
\begin{equation*}
Y_t^{t,\zeta;v}+\varepsilon=\sum_{j=1}^N(Y_t^{t,x_j;v}+\varepsilon)I_{A_j}\geq\sum_{j=1}^Nu(t,x_j)I_{A_j}=u(t,\zeta).
\end{equation*}
That is to say that both (\ref{pz30}) and (\ref{pz31}) are satisfied for the simple random state variables.

For any random state variable $\eta\in L^{2}(\Omega,\mathscr{F}_t;\mathbb{R}^n)$, there exists a sequence of simple random variables $\{\zeta\}_{n=1}^{\infty}$ such that $$\lim\limits_{n\to\infty}|\zeta_n-\eta|=0.$$
By Proposition \ref{pz27} and Lemma \ref{pz26} we have
%
for any $v\in\mathcal{U}$,
\begin{eqnarray*}
\lim_{n\to\infty}|Y_t^{t,\zeta_{n};v}-Y_t^{t,\eta;v}|=0\ \ \text{a.s.}\ \ \ \text{and}\ \ \ \lim_{n\to\infty}|u(t,\zeta_{n})-u(t,\eta)|=0\ \ \text{a.s.}
\end{eqnarray*}
Since $Y_t^{t,\zeta_{n};v}\leq u(t,\zeta_{n})$ holds for all $n$, (\ref{pz30}) follows for $\eta$ as $n\to\infty$.

Also (\ref{pz31}) is true for any random state variable $\eta\in L^{2}(\Omega,\mathscr{F}_t;\mathbb{R}^n)$. To demonstrate this, we choose a simple random variable $\zeta$ such that $|\zeta-\eta|<{\varepsilon\over{3C}}$. In view of Proposition \ref{pz27} and Lemma \ref{pz26} again it yields that for any $v\in\mathcal{U}^t$,
\begin{eqnarray*}
|Y_t^{t,\zeta;v}-Y_t^{t,\eta;v}|\leq{\varepsilon\over3}\ \ \text{and}\ \ |u(t,\zeta)-u(t,\eta)|\leq{\varepsilon\over3}.
\end{eqnarray*}
Note that since $\zeta$ is a simple random variable there exists an admissible control $\tilde{v}\in\mathcal{U}$ such that
\begin{equation*}
Y_t^{t,\zeta;\tilde{v}}+{\varepsilon\over3}\geq u(t,\zeta).
\end{equation*}
Hence
\begin{equation*}
Y_t^{t,\eta;\tilde{v}}\geq-|Y_t^{t,\zeta;\tilde{v}}-Y_t^{t,\eta;\tilde{v}}|+Y_t^{t,\zeta;\tilde{v}}\geq u(t,\zeta)-{{2\varepsilon}\over3}
\geq u(t,\eta)-\varepsilon,
\end{equation*}
which puts an end of proof for Lemma \ref{pz29}.
\end{proof}

Now we are well prepared to prove the dynamic programming principle in our settings.
 \begin{theorem}{\bf{(Dynamic programming principle with non-Lipschitz aggregator)}}\label{dp}
Assume 
Conditions (H1)--(H6). Then for any $0 \leq \delta \leq T-t$, the value function $u(t,x)$ has the following property:
 \begin{equation*}\label{eq25}
 u(t,x)= \esssup_{v\in \mathcal{U}}G_{t,t+\delta}^{t,x;v}[u(t+\delta, X_{t+\delta}^{t,x;v})].
 \end{equation*}
 \end{theorem}
\begin{proof}
First of all, by the definition of notation and the uniqueness of solution of BSDE (\ref{eq5a}) we have
 \begin{equation*}\label{eq26}
 \begin{aligned}
 u(t,x)&=\esssup_{v\in \mathcal{U}}G_{t,T}^{t,x;v}[h(X_T^{t,x;v})]\\
 &=\esssup_{v\in \mathcal{U}}G_{t,t+\delta}^{t,x;v}[Y_{t+\delta}^{t,x;v}]\\
&=\esssup_{v\in \mathcal{U}}G_{t,t+\delta}^{t,x;v}[Y_{t+\delta}^{t+\delta, X_{t+\delta}^{t,x;v};v}].
\end{aligned}
\end{equation*}

Then we need use the comparison theorem of BSDE in the next step. Bear in mind that the aggregator of BSDE (\ref{eq5a}) is not Lipschitz with respect to the first unknown variable, so the classical comparison theorem does not work. Instead, we apply the generalized comparison theorem with  ``weakly" monotonic aggregator (Theorem \ref{pz2}) to our case. But before the application of this generalized comparison theorem we need first show that $u(s,X_s^{t,x;v})$ for $t\leq s\leq T$ and $v\in\mathcal{U}$ is square integrable which acts as the terminal value of BSDE. To see this, note that for any $\varepsilon>0$, by Lemma \ref{pz29} there exists $\tilde{v}\in\mathcal{U}$ such that $$Y_s^{s,X_s^{t,x;v};\tilde{v}}\leq u(s,X_s^{t,x;v})\leq Y_s^{s,X_s^{t,x;v};\tilde{v}}+\varepsilon,$$ so we only need to prove
$E[|Y_s^{s,X_s^{t,x;v};\tilde{v}}|^2]<\infty$. Noticing the uniform boundedness of the control processes in $\mathcal{U}$ we use Propositions \ref{pz24} and \ref{pz27} to know that
\begin{eqnarray*}
E[|Y_s^{s,X_s^{t,x;v};\tilde{v}}|^2]&\leq&C_p\big(1+E[|X_s^{t,x;v}|^{2}+\int_s^T|f(r,0,0,0,\tilde{v}_r)|^2dr]\big)\nonumber\\
&\leq&C_p\big(1+E[|x|^{2}+\int_t^T|v_r|^2dr+\int_s^T|f(r,0,0,0,\tilde{v}_r)|^2dr]\big)<\infty.
\end{eqnarray*}
Hence the application of the Theorem \ref{pz2}, together with the definition of the value function, yields
 \begin{equation}\label{eq28}
 u(t,x)\leq \esssup_{v\in \mathcal{U}} G_{t,t+\delta}^{t,x;v}[u(t+\delta,X_{t+\delta}^{t,x;v})].
 \end{equation}

On the other hand, according to Lemma \ref{pz29}, for arbitrary $\varepsilon$, there exists an admissible control $\bar{v}\in \mathcal{U}$ such that
 \begin{equation}\label{eq29}
 u(t+\delta,X_{t+\delta}^{t,x;v}) \leq Y_{t+\delta} ^{t+\delta,X_{t+\delta}^{t,x;v}; \bar{v}}+\varepsilon.
 \end{equation}
Hence we have
  \begin{equation}\label{eq30}
  \begin{aligned}
  u(t,x)&\geq\esssup_{v\in \mathcal{U}}G_{t,t+\delta}^{t,x;v}[Y_{t+\delta} ^{t+\delta,X_{t+\delta}^{t,x;v};\bar{v}}]\\
  &\geq\esssup_{v\in \mathcal{U}}G_{t,t+\delta}^{t,x;v}[u(t+\delta,X_{t+\delta}^{t,x;v})-\varepsilon]\\
  & \geq \esssup_{v\in \mathcal{U}}G_{t,t+\delta}^{t,x;v}[u(t+\delta,X_{t+\delta}^{t,x;v})]- \sqrt{C_p}\varepsilon,
\end{aligned}
\end{equation}
with a constant $C_p$. 
Here
  the second inequality in \eqref{eq30} is based on \eqref{eq29} and the comparison theorem, and the last inequality comes from a basic estimate of BSDE. To see this,
we set $Y_t^{1;v}=G_{t,t+\delta}^{t,x;v}[u(t+\delta,X_{t+\delta}^{t,x;v})]$ and $Y_t^{2;v}=G_{t,t+\delta}^{t,x;v}[u(t+\delta,X_{t+\delta}^{t,x;v})-\varepsilon]$. Applying It\^{o}'s formula to ${\rm e}^{-ks}|Y_s^{1;v}-Y_s^{2;v}|^2$, where $t\leq s\leq t+\delta$ and $k>0$ is a sufficiently large constant, we have
  \begin{equation*}
  |Y_t^{1;v}-Y_t^{2;v}|^2 \leq C_pE^{\mathscr{F}_t} [|u(t+\delta,X_{t+\delta}^{t,x;v})-\varepsilon-u(t+\delta,X_{t+\delta}^{t,x;v})|^2]=C_pE^{\mathscr{F}_t} [\varepsilon^2].
  \end{equation*}
  Thus
  \begin{equation*}
  \esssup_{v\in \mathcal{U}} Y_t^{1;v}-\esssup_{v\in \mathcal{U}} Y_t^{2;v} \leq\esssup_{v\in \mathcal{U}}(Y_t^{1;v}-Y_t^{2;v}) \leq\esssup_{v\in \mathcal{U}}  |Y_t^{1;v}-Y_t^{2;v}| \leq \sqrt{C_p}\varepsilon,
  \end{equation*}
  which implies
  \begin{equation*}
  \esssup_{v\in \mathcal{U}} Y_t^{1;v}\leq \esssup_{v\in \mathcal{U}} Y_t^{2;v}+ \sqrt{C_p}\varepsilon,
  \end{equation*}
  i.e.
  \begin{equation*}
  \esssup_{v\in \mathcal{U}} G_{t,t+\delta}^{t,x;v}[u(t+\delta,X_{t+\delta}^{t,x;v})] \leq  \text{esssup}_{v\in \mathcal{U}} G_{t,t+\delta}^{t,x;v}[u(t+\delta,X_{t+\delta}^{t,x;v})-\varepsilon]+ \sqrt{C_p}\varepsilon.
  \end{equation*}

Therefore, the dynamic programming follows from \eqref{eq28} and \eqref{eq30}, due to the arbitrariness of $\varepsilon$ in \eqref{eq30}.
 \end{proof}

\section{Viscosity solution of HJB equation}
We aim to establish the connection in this section between the value function of our concerned stochastic recursive control problem and the viscosity solution of its corresponding HJB equation. For this, we need to assume that the aggregator of BSDE in our concerned recursive control problem is independent of the second unknown variable throughout Section 4, i.e. $f(t,x,y,z,v)=f(t,x,y,v)$ for $f$ in BSDE (\ref{eq5a}).

In this situation the HJB equation, a second-order fully nonlinear PDE of parabolic type, has a form:
\begin{numcases}{}\label{eq31}
{{\partial}\over{\partial t}}u+ H(t,x,u,D_{x}u, D^2_{x}u)=0,\ \ \ \ \ \ (t,x) \in [0,T) \times \mathbb{R}^n,\nonumber\\
u(T,x)=h(x).
\end{numcases}
Here $D_{x}u$  and $D^2_{x}u$ denote the gradient matrix and the Hessian matrix of $u$, respectively.  The Hamiltonian $H=H(t,x,r,p,A):[0,T]\times \mathbb{R}^n  \times \mathbb{R}^1 \times \mathbb{R}^n \times \mathbb{S}^n \rightarrow \mathbb{R}^1$ is defined as below:
 \begin{equation*}\label{eq33}
   H\triangleq\sup_{v \in U} \{\frac{1}{2}Tr(\sigma(t,x,v)\sigma^*(t,x,v)A)+ \langle p,b(t,x,v)\rangle+ f(t,x,r,v)\},
  \end{equation*}
where $\mathbb{S}^n$ is the matrix space including all $n \times n$ symmetric matrices.

Denote by $C^{1,2}([0,T]\times \mathbb{R}^n;\mathbb{R}^1)$ the space of all functions from $[0,T]\times \mathbb{R}^n$ to $\mathbb{R}^1$ whose derivatives up to the first order with respect to time variable and up to the second order with respect to state variable are continuous. Then we recall the definition for the viscosity solution of HJB equation \eqref{eq31}.
\begin{definition}
A continuous function $u:[0,T] \times \mathbb{R}^n\to\mathbb{R}^1$ is a viscosity subsolution (resp. supersolution) of HJB equation \eqref{eq31}, if for any $x \in \mathbb{R}^n$,  $u(T,x)  \leq h(x)$ (resp. $u(T,x)  \geq h(x)$), and for any $\varphi \in C^{1,2}([0,T]\times \mathbb{R}^n;\mathbb{R}^1)$, $(t,x) \in [0,T) \times \mathbb{R}^n$, $\varphi-u$ attains a global minimum (resp. maximum) at $(t,x)$ and $\varphi$ satisfies
  \begin{eqnarray*}
  && {{\partial}\over{\partial t}}\varphi(t,x) + H(t,x,\varphi, D_{x}\varphi, D^2_{x}\varphi) \geq 0 \\ \label{eq35}
  && \big(resp. ~~{{\partial}\over{\partial t}}\varphi(t,x) + H(t,x,\varphi, D_{x}\varphi, D^2_{x}\varphi) \leq 0\big).\label{eq36}
  \end{eqnarray*}
We call $u$ the viscosity solution of \eqref{eq31} if $u$ is both a viscosity subsolution and a viscosity supersolution.
\end{definition}
We need some preliminaries to establish the connection. First, we indicate the continuity of the value function.
\begin{prop}\label{propofcontinous}
Assume 
Conditions (H1)--(H6). Then the value function $u(t,x):[0,T]\times \mathbb{R}^n\to\mathbb{R}^1$ defined in (\ref{eq10}) is continuous with respect to $(t,x)$.
\end{prop}
Note that the Lipschitz continuity of $u(t,x)$ with respect to $x$ is a result of Lemma \ref{pz26} which is the counterpart of Lemma 5.2 in \cite{Peng}. Also we can prove the ${1\over2}$-H\"{o}lder continuity of $u(t,x)$ with respect to $t$ in a similar way referring to Proposition 5.5 in \cite{Peng}, which together with Lemma \ref{pz26} implies the continuity of the value function with respect to $(t,x)$. There is nothing special for the non-Lipschitz aggregator in our setting in comparison with the classical Lipschitz aggregator, so we leave out the proof here. 

\vspace{2mm}
Then we define a sequence of smootherized functions $f_n$, $n\in\mathbb{N}$,  based on the aggregator $f$ as follows:
\begin{equation}\label{eq60}
f_n(t,x,y,v)\triangleq \ (\rho_n * f(t,x,\cdot,v))(y),
\end{equation}
where $\rho_n: \mathbb{R}^1\rightarrow \mathbb{R}^+ $, $n\in\mathbb{N}$, is a family of sufficiently smooth functions with the compact support in $[-\frac{1}{n}, \frac{1}{n}]$ and satisfies
\begin{equation*}\label{eq61}
\int_{\mathbb{R}^1}\rho_n(a)da=1.
\end{equation*}
Consequently, we have a sequence of BSDEs with the smootherized aggregators $f_n$, $n \in \mathbb{N}$, on the interval $[t,T]$:
\begin{equation}\label{eq66}
\begin{aligned}
Y^{t,x,n;v}_s= h(X_T^{t,x;v}) + \int_s^T f_n(r, X_r^{t,x;v},Y^{t,x,n;v}_r ,v_r )dr-\int_s^T Z^{t,x,n;v}_r dB_r.
\end{aligned}
\end{equation}
With the solutions of BSDEs (\ref{eq66}), we can define a sequence of stochastic recursive control problems 
whose cost functional for each $n\in\mathbb{N}$ is
\begin{eqnarray*}\label{eq67}
& J_n(t,x;v)\triangleq  Y_t^{t,x,n;v}\ \ \ \ {\rm for}\ v\in \mathcal{U},\ t\in[0,T],\ x\in\mathbb{R}^n
\end{eqnarray*}
and corresponding control problem is to find an optimal $\bar{v}\in\mathcal{U}$ to maximize above cost functional for given $(t,x)$.
Thus, for each $n\in\mathbb{N}$, the value function of control problem has a form as
\begin{eqnarray}\label{pz18}
& u_n(t,x)\triangleq \text{esssup}_{v\in \mathcal{U}}J_n(t,x;v)\ \ \ \ {\rm for}\ t\in[0,T],\ x\in\mathbb{R}^n
\end{eqnarray}
and the Hamiltonian appears like
 \begin{equation}\label{pz20}
   H_n(t,x,r,p,A)\triangleq\sup_{v \in U} \{\frac{1}{2}Tr(\sigma(t,x,v)\sigma^*(t,x,v)A)+ \langle p,b(t,x,v)\rangle+ f_n(t,x,r,v)\},
  \end{equation}
where $(t,x,r,p,A)\in [0,T]\times \mathbb{R}^n  \times \mathbb{R}^1 \times \mathbb{R}^n \times \mathbb{S}^n$.


Then we prove the uniform convergence of the smootherized aggregators in a compact subset of their domain utilizing the continuity of the aggregator. 
\begin{lemma}\label{lemma0}
Assume Conditions (H3)--(H4). Then $f_n$ defined in (\ref{eq60}) converges to $f$, uniformly in every compact subset of $[0,T]\times \mathbb{R}^n\times\mathbb{R}^1\times U$.
\end{lemma}
\begin{proof}
Since $\int_{\mathbb{R}^1} \rho_n(a)da=1$,
we have
 \begin{equation*}\label{eq78}
 f_n(t,x,y,v)-f(t,x,y,v)=\int_{\mathbb{R}^1}(f(t,x,y-a,v)-f(t,x,y,v))\rho_n(a)da.
 \end{equation*}
For any given compact set $K\subset[0,T]\times \mathbb{R}^n\times\mathbb{R}^1\times U$, there exists another compact set $\hat{K}$ such that $(t,x,y-a,v)\in \hat{K}$ for any $(t,x,y,v)\in K$ and $a\in[-1, 1]$. Notice that since $f(t,x,y,v)$ is continuous with respect to $(t,y,v)$ and Lipschitz continuous with respect to $x$, we know the continuity and further the uniform continuity of $f(t,x,y,v)$ with respect to $(t,x,y,v)$ in the compact set $\hat{K}$. So, for any $\varepsilon >0$, as $n$ is sufficiently large we have
 \begin{equation*}\label{eq79}
\begin{aligned}
& \sup_{(t,x,y,v) \in K} |f_n(t,x,y,v)-f(t,x,y,v)|\\
& \leq \sup_{(t,x,y,v)\in K}\int_{|a| \leq \frac{1}{n}}|f(t,x,y-a,v)-f(t,x,y,v)|\rho_n(a)da \\
& \leq \varepsilon \int_{|a| \leq \frac{1}{n}} \rho_n(a)da\\
&=\varepsilon,
\end{aligned}
\end{equation*}
which implies the desired conclusion.
\end{proof}

As a result, we can further get the uniform convergence of the solutions of BSDEs with smootherized aggregators in $L^2(\Omega)$ space.
\begin{lemma}\label{lemma1}
Assume 
Conditions (H1)--(H6). Then for any $v\in \mathcal{U}$,
$$\lim_{n\to\infty}\sup_{(t,x)\in K}E[|{Y}^{t,x,n;v}_t -Y_t^{t,x;v}|^2]=0,$$
where $K$ is an arbitrary compact set in $[0,T]\times \mathbb{R}^n$, $Y^{t,x;v}_\cdot$ and $Y^{t,x,n;v}_\cdot$ are the solutions of BSDEs \eqref{eq5a} and \eqref{eq66}, respectively.
\end{lemma}
\begin{proof} 
Firstly, it is obvious that the smootherized aggregator $f_n$ satisfies Conditions (H3)--(H6). Hence, applying It\^{o}'s formula to $|{Y}^{t,x,n;v}_s -Y_s^{t,x;v}|^2$, we have for any $(t,x)\in K$,
\begin{equation*}
\begin{aligned}
&E[|{Y}^{t,x,n;v}_t -Y_t^{t,x;v}|^2]\\
\leq &C_pE\Big[\int_t^T|f_n(s,X_s^{t,x;v},Y_s^{t,x;v},v_s)-f(s,X_s^{t,x;v},Y_s^{t,x;v},v_s)|^2 ds\Big]\\
= &C_pE [\int_t^T|f_n-f|^21_{\{\{\sup\limits_{s\in[t,T]}|X_s^{t,x;v}| \geq N\} \cup \{\sup\limits_{s\in[t,T]}|Y_s^{t,x;v}| \geq N \}\}}ds] \\
& +C_pE[\int_t^T|f_n-f|^21_{\{\{\sup\limits_{s\in[t,T]}|X_s^{t,x;v}| < N\}\cap \{\sup\limits_{s\in[t,T]}|Y_s^{t,x;v}| < N\}\}}ds]\\
\leq & C_pE[\int_t^T|f_n-f|^21_{\{\sup\limits_{s\in[t,T]}|X_s^{t,x;v}| \geq N\}}ds]+ C_pE[\int_t^T|f_n-f|^21_{\{\sup\limits_{s\in[t,T]}|Y_s^{t,x;v}| \geq N\}}ds] \\
 &+ C_pE [\int_t^T|f_n-f|^21_{\{\{\sup\limits_{s\in[t,T]}|X_s^{t,x;v}| < N\}  \cap \{\sup\limits_{s\in[t,T]}|Y_s^{t,x;v}| < N\}\}}ds].
\end{aligned}
\end{equation*}
Then we define
\begin{eqnarray*}
&&J_1\triangleq E[\int_t^T|f_n(s,X_s^{t,x;v},Y_s^{t,x;v},v_s)-f(s,X_s^{t,x;v},Y_s^{t,x;v},v_s)|^21_{\{\sup\limits_{s\in[t,T]}|X_s^{t,x;v}| \geq N\}}ds],\\
&&J_2\triangleq E[\int_t^T|f_n(s,X_s^{t,x;v},Y_s^{t,x;v},v_s)-f(s,X_s^{t,x;v},Y_s^{t,x;v},v_s)|^21_{\{\sup\limits_{s\in[t,T]}|Y_s^{t,x;v}| \geq N\}}ds],\\
&&J_3\triangleq E [\int_t^T|f_n(s,X_s^{t,x;v},Y_s^{t,x;v},v_s)-f(s,X_s^{t,x;v},Y_s^{t,x;v},v_s)|^2\\
&&\ \ \ \ \ \ \ \ \ \ \ \ \ \ \ \times1_{\{\{\sup\limits_{s\in[t,T]}|X_s^{t,x;v}| < N\}  \cap \{\sup\limits_{s\in[t,T]}|Y_s^{t,x;v}| < N\}\}}ds],
\end{eqnarray*}
and deal with $J_1$, $J_2$ and $J_3$ in turn.

For $J_1$, it turns out that
\begin{eqnarray*}\label{pz38}
&&\sup_{(t,x)\in K}J_1\nonumber\\
&\leq&\sup_{(t,x)\in K}2E [\int _t^T\big(|f_n(s,X_s^{t,x;v},Y_s^{t,x;v},v_s)|^2+|f(s,X_s^{t,x;v},Y_s^{t,x;v},v_s)|^2\big)1_{\{\sup\limits_{s\in[t,T]}|X_s^{t,x;v}| \geq N\}}ds]\nonumber\\
&\leq &  \sup_{(t,x)\in K}C_pE [\int_t^T(1+|X_s^{t,x;v}|^2+ |Y_s^{t,x;v}|^{2p})1_{\{\sup\limits_{s\in[t,T]}|X_s^{t,x;v}| \geq N\}}ds]\nonumber\\
&\leq &  \sup_{(t,x)\in K}C_p\big(E [\int_t^T(1+|X_s^{t,x;v}|^4+ |Y_s^{t,x;v}|^{4p})ds]\big)^{1\over2}\big(\sup_{(t,x)\in K}P[\sup\limits_{s\in[t,T]}|X_s^{t,x;v}| \geq N]\big)^{1\over2}.
\end{eqnarray*}
To estimate above, we use Chebychev's inequality and Proposition \ref{pz24} to obtain for any $N>0$,
 \begin{equation*}
 P[\sup_{s\in[t,T]}|X_s^{t,x;v}| \geq N ] \leq  \frac{1}{N^2}E[\sup_{s\in[t,T]} |X_s^{t,x;v}|^2]\leq \frac{C_p}{N^2}|x|^2.
 \end{equation*}
Thus for any $\delta>0$, it follows from the boundedness of $x$ in $K$ that with a sufficiently large $N$,
 \begin{equation}\label{pz4}
 \sup_{(t,x)\in K}P[\sup_{s\in[t,T]}|X_s^{t,x;v}| \geq N ] \leq \delta.
\end{equation}
Moreover, by Propositions \ref{pz24} and \ref{pz27}, we know
\begin{eqnarray*}
\sup_{(t,x)\in K}\big(E [\int_t^T(1+|X_s^{t,x;v}|^4+ |Y_s^{t,x;v}|^{4p})ds]\big)^{1\over2}<\infty,
\end{eqnarray*}
which together with (\ref{pz4}) implies that for any given $\varepsilon>0$, there exists a sufficiently large ${N}_1$ such that as $N\geq {N}_1$, for all $n\in\mathbb{N}$,
$$
J_1\leq \varepsilon,\ \ \text{uniformly in the compact set $K$.}
$$

Then we turn to $J_2$, and by Proposition \ref{pz27} 
we have
\begin{equation*}
\begin{aligned}
&  E[\sup_{s\in[t,T]}|Y_s^{t,x;v}|^{2p}]
 \leq C_p(1+|x|^{2p}).
\end{aligned}
\end{equation*}
Again with a sufficiently large $N$, the application of Chebychev's inequality in $K$ leads to for any $\delta>0$,
\begin{equation}\label{pz5}
\sup_{(t,x)\in K}P[\sup_{s\in[t,T]}|Y_s^{t,x;v}| \geq N] \leq \delta.
\end{equation}
Similar to the treatment of $J_1$, by (\ref{pz5}) we can find a sufficiently large ${N}_2$ such that as $N\geq {N}_2$, for all $n\in\mathbb{N}$,
$$
J_2\leq \varepsilon,\ \ \text{uniformly in the compact set $K$.}
$$
We take $\hat{N}=N_1\vee N_2$ and use $\hat{N}$ to prove the uniform convergence of $J_3$.

To this end, notice for any $v\in \mathcal{U}$,
\begin{eqnarray*}
\sup_{(t,x)\in K}J_3&\leq&E [\int_t^T\sup_{(t,x)\in K}|f_n-f|^21_{\{\{\sup\limits_{s\in[t,T]}|X_s^{t,x;v}| < \hat{N}\}  \cap \{\sup\limits_{s\in[t,T]}|Y_s^{t,x;v}| < \hat{N}\}\}}ds]\\
&\leq&\int_t^T\sup_{(t,x,y)\in [0,T]\times\bar{B}^n_{\hat{N}}\times[-\hat{N},\hat{N}]}|f_n(t,x,y,v)-f(t,x,y,v)|^2ds,
\end{eqnarray*}
where $\bar{B}^n_{\hat{N}}$ is the closed ball with the radium $\hat{N}$ in $\mathbb{R}^n$.
By the dominated convergence theorem and Lemma \ref{lemma0}, as $n$ is sufficiently large, we have
$$
J_3\leq \varepsilon,\ \ \text{uniformly in the compact set $K$.}
$$

Therefore, due to the arbitrariness of $\varepsilon$, the claim that
$
\lim\limits_{n\to\infty}\sup\limits_{(t,x)\in K}E[|{Y}^{t,x,n;v}_t -Y_t^{t,x;v}|^2] = 0 
$
follows, which puts an end of proof.
\end{proof}

The uniform convergence in compact subset of domain holds for the value function as well, which is displayed in next lemma.
 \begin{lemma}\label{lemma2}
Assume 
Conditions (H1)--(H6). Then $u_n$ converges to $u$, uniformly in every compact subset of $[0,T]\times \mathbb{R}^n$.
 \end{lemma}
\begin{proof}
Given arbitrary $\varepsilon > 0$, for any $t\in[0,T]$ and $x \in \mathbb{R}^n$,  we can find ${v_1}\in \mathcal{U}$ such that
 \begin{equation*}\label{eq71}
 u(t,x) < Y_t^{t,x;v_1} + \varepsilon.
 \end{equation*}
So
 \begin{equation}\label{eq72}
  u(t,x)-u_n(t,x) = u(t,x)-\sup_{v\in \mathcal{U}} Y_t^{t,x,n;v}\leq Y_t^{t,x;v_1} + \varepsilon - {Y}^{t,x,n;v_1}_t.
 \end{equation}
On the other hand, for above $\varepsilon$, there exists ${v_2}\in \mathcal{U}$ such that
 \begin{equation*}\label{eq73}
 u_n(t,x) \leq {Y}^{t,x,n;v_2}_t+\varepsilon,
 \end{equation*}
which implies
 \begin{equation}\label{eq74}
 u(t,x)-u_n(t,x) \geq u(t,x) - {Y}^{t,x,n;v_2}_t - \varepsilon \geq Y^{t,x;v_2}_t -{Y}^{t,x,n;v_2}_t -\varepsilon.
 \end{equation}
Since \eqref{eq72} and \eqref{eq74}, we have
\begin{equation*}
|u(t,x)-u_n(t,x)|\leq E[|Y_t^{t,x;v_1}- {Y}^{t,x,n;v_1}_t|]+E[|Y^{t,x;v_2}_t -{Y}^{t,x,n;v_2}_t|]+ 2\varepsilon.
 \end{equation*}

Noticing $v_1$ and $v_2$ are given admissible controls, by Lemma \ref{lemma1} we know that for any compact set $K\subset[0,T]\times \mathbb{R}^n$,  $$\lim_{n\to\infty}\sup_{(t,x)\in K}E[|Y_t^{t,x;v_1}- {Y}^{t,x,n;v_1}_t|]+E[|Y^{t,x;v_2}_t -{Y}^{t,x,n;v_2}_t|]=0.$$
Due to the arbitrariness of $\varepsilon$, we obtain the uniform convergence of the value functions $u_n$ to $u$ in $K$. 
\end{proof}

We have known that $f_n$ converges uniformly to $f$ in the compact subset of their domain, so, by definition of the Hamiltonian, it comes without a surprise that the same kind of convergence of the Hamiltonian $H_n$ to $H$ holds as well.
\begin{lemma}\label{lemma3}
Assume 
Conditions (H1)--(H6). Then $H_n$ converges to $H$, uniformly in every compact subset of their domain.
\end{lemma}
\begin{proof}
To see this, for any $t\in[0,T]$, $x\in\mathbb{R}^n$, $r\in\mathbb{R}^1$, $p\in\mathbb{R}^n$, $A\in\mathbb{S}^n$, $v\in U$, we set $$\mathcal{A}=\frac{1}{2}Tr(\sigma(t,x,v)\sigma^*(t,x,v)A)+ \langle p,b(t,x,v)\rangle+ f_n(t,x,r,v)$$ and
$$\mathcal{B}=\frac{1}{2}Tr(\sigma(t,x,v)\sigma^*(t,x,v)A)+ \langle p,b(t,x,v)\rangle+ f(t,x,r,v).$$
Noticing
\begin{equation*}
H_n-H=\sup_{v \in U} \mathcal{A}-\sup_{v \in U} \mathcal{B} \leq \sup_{v \in U} (\mathcal{A}-\mathcal{B})=\sup_{v \in U} (f_n-f) \leq  \sup_{v \in U} |f_n-f|
\end{equation*}
and
\begin{equation*}
H-H_n =\sup_{v \in U} \mathcal{B}-\sup_{v \in U} \mathcal{A} \leq \sup_{v \in U} (\mathcal{B}-\mathcal{A})=\sup_{v \in U} (f-f_n) \leq  \sup_{v \in U} |f-f_n|,
\end{equation*}
we have
\begin{equation*}
|H_n-H| \leq  \sup_{v \in U} |f_n(t,x,r,v)-f(t,x,r,v)|.
\end{equation*}

Note that for any compact set $K\subset[0,T] \times \mathbb{R}^n \times \mathbb{R}^1\times \mathbb{R}^n\times\mathbb{S}^n$ and $(t,x,r,p,A,v)\in K\times U$, $(t,x,r,v)\in\hat{K}$, where $\hat{K}$ is a compact set in  $[0,T]\times \mathbb{R}^n\times\mathbb{R}^1\times U$.
Hence by Lemma \ref{lemma0}, we have
\begin{eqnarray*}
&&\lim_{n\to\infty}\sup_{(t,x,r,p,A) \in K} |H_n(t,x,r,p,A)-H(t,x,r,p,A)|\\
&\leq&\lim_{n\to\infty}\sup_{(t,x,r,p,A) \in K}\sup_{v \in U}|f_n(t,x,r,v)-f(t,x,r,v)|\\
 &\leq&\lim_{n\to\infty}\sup_{(t,x,y,v) \in \hat{K}} |f_n(t,x,y,v)-f(t,x,y,v)|=0.
\end{eqnarray*}
Therefore, the uniform convergence of $H_n$ to $H$ in $K$ follows from above.
\end{proof}

To end the preliminaries, we introduce the stability property of viscosity solutions below (see e.g. Lemma 6.2 in Fleming and Soner \cite{Fleming} for details of proof) which provides a method based on the uniform convergence of Hamiltonians to get the connection between the value function and the solution of HJB equation.
\begin{prop}{\bf {(Stability)}}\label{prop1}
Let $u_n$ be a viscosity subsolution (resp. supersolution) to the following PDE
\begin{equation*}
\frac{\partial}{\partial t} u_n(t,x) + H_n(t,x,u_n(t,x), D_{x}u_n(t,x), D^2_{x}u_n(t,x) )=0,\quad (t,x)\in [0,T)\times \mathbb{R}^n,
\end{equation*}
where $H_n(t,x,r,p,A):[0,T]\times \mathbb{R}^n  \times \mathbb{R}^1 \times \mathbb{R}^n \times \mathbb{S}^n \rightarrow \mathbb{R}^1$ is continuous
and satisfies the ellipticity condition
\begin{equation}\label{pz37}
H_n(t,x,r,p,X) \leq H_n(t,x,r,p,Y) \ \ \ \ \text{whenever }~X \leq Y.
\end{equation}
Assume that $H_n$ and $u_n$ converge to $H$ and $u$, respectively, uniformly in every compact subset of their own domains. Then $u$ is a viscosity subsolution (resp. supersolution)  of the limit equation
\begin{equation*}
\frac{\partial}{\partial t} u(t,x)+H(t,x,u(t,x), D_{x}u(t,x), D^2_{x}u(t,x) )=0.
\end{equation*}
\end{prop}



Now we are well prepared to prove the main theorem in this section.
\begin{theorem}\label{th2}
  Assume 
Conditions (H1)--(H6). Then $u$ defined in \eqref{eq10} is a viscosity solution of HJB equation \eqref{eq31}.
  \end{theorem}
\begin{proof} We divide our proof into two steps.

Step 1. Assume that $|f(t,x,0,v)|$ is uniformly bounded for any $(t,x,v)\in[0,T]\times \mathbb{R}^n\times U$.

Note that the uniform boundedness of $|f(t,x,0,v)|$ implies the global Lipschitz of $f_n(t,x,y,v)$ with respect to $y$. To see this, for any $(t,x,v)\in[0,T]\times \mathbb{R}^n\times U$, $y_1,y_2\in\mathbb{R}^1$, by (H6) it yields that
\begin{eqnarray*}
&&|f_n(t,x,y_1,v)-f_n(t,x,y_2,v)|\nonumber\\
&=&|\int_{|a|\leq{1\over n}}f(t,x,a,v)\big(\rho_n(y_1-a)-\rho_n(y_2-a)\big)da|\nonumber\\
&\leq&\max_{a\in[-{1\over n},{1\over n}]}|f(t,x,a,v)|\int_{|a|\leq{1\over n}}|\rho_n(y_1-a)-\rho_n(y_2-a)|da\nonumber\\
&\leq&\max_{a\in[-{1\over n},{1\over n}]}\big(|f(t,x,0,v)|+\kappa(1+ |a|^p)\big)\int_{|a|\leq{1\over n}}C_p(n)|y_1-y_2|da\nonumber\\
&\leq&C_p(\kappa,n)|y_1-y_2|.
\end{eqnarray*}
Hence we immediately know from Theorem 7.3 in \cite{Peng} that $u_{n}(t,x)$ is the viscosity solution of the following equations:
\begin{numcases}{}\label{pz17}
{{\partial}\over{\partial t}}u_n+ H_n(t,x,u_n,D_{x}u_n, D^2_{x}u_n)=0,,\ \ \ \ \ \ (t,x) \in [0,T) \times \mathbb{R}^n,\nonumber\\
u_n(T,x)=h(x),
\end{numcases}
where $u_n$ and $H_n$ defined in (\ref{pz18}) and (\ref{pz20}), respectively.

By Lemmas \ref{lemma0}--\ref{lemma3} the uniform convergence of $f_{n}$ to $f$, $Y_t^{t,x,n;v}$ to $Y_t^{t,x;v}$, $u_{n}$ to $u$ and $H_{n}$ to $H$ holds in every compact subset of their own domains as $n\to\infty$. Moreover, $H_n$ satisfies the ellipticity condition (\ref{pz37}). Therefore,
by the stability of viscosity solution stated in Proposition \ref{prop1}, we know that $u$ is a viscosity solution of the limit equation
\begin{equation*}
\frac{\partial}{\partial t} u(t,x)+H(t,x,u(t,x), D_{x}u(t,x), D^2_{x}u(t,x) )=0.
\end{equation*}
As for the terminal value of above equation, i.e. $u(T,x)=h(x)$, which can be seen from the definition of the value function. Thereby $u$ is a solution of HJB equation (\ref{eq31}). \\ 

Step 2. $|f(t,x,0,v)|$ is not necessary to be uniformly bounded for any $(t,x,v)\in[0,T]\times \mathbb{R}^n\times U$.

We construct a sequence of functions
$$f_m(t,x,y,v)\triangleq f(t,x,y,v)-f(t,x,0,v)+\Pi_m\big(f(t,x,0,v)\big)\ \ \ \ \text{for}\ m\in\mathbb{N},$$
where $\Pi_m(x)={\inf(m,|x|)\over|x|}x$.

With these $f_m$, we get a family of BSDEs for $m\in\mathbb{N}$ on the interval $[t,T]$:
\begin{equation*}
\begin{aligned}
{Y}^{t,x,m;v}_t= h(X_T^{t,x;v}) + \int_t^T f_{m}(s, X_s^{t,x;v},{Y}^{t,x,m;v}_s ,v_s )ds-\int_t^T {Z}^{t,x,m;v}_s dB_s.
\end{aligned}
\end{equation*}
Similarly we define the corresponding cost functional
\begin{eqnarray*}
J_{m}(t,x;v)\triangleq  Y_t^{t,x,m;v}\ \ \ \ {\rm for}\ v\in \mathcal{U},\ t\in[0,T],\ x\in\mathbb{R}^1,
\end{eqnarray*}
the value function
\begin{eqnarray*}
u_{m}(t,x)\triangleq \text{esssup}_{v\in \mathcal{U}}J_{m}(t,x;v)\ \ \ \ {\rm for}\ t\in[0,T],\ x\in\mathbb{R}^1,
\end{eqnarray*}
and the Hamiltonian
\begin{equation*}
   H_m(t,x,r,p,A)\triangleq\sup_{v \in U} \{\frac{1}{2}Tr(\sigma(t,x,v)\sigma^*(t,x,v)A)+ \langle p,b(t,x,v)\rangle+ f_m(t,x,r,v)\}
\end{equation*}
for $(t,x,r,p,A)\in [0,T]\times \mathbb{R}^n  \times \mathbb{R}^1 \times \mathbb{R}^n \times \mathbb{S}^n$.

Since $f_m(t,x,0,v)=\Pi_m\big(f(t,x,0,v)\big)$, $f_m(t,x,0,v)$ is uniformly bounded. Moreover, it is not difficult to verify that $f_m$ satisfies Conditions (H3)--(H6). 
Hence $f_m$ satisfies the conditions in Step 1. By Step 1 we know that $u_m$ is a viscosity solution of the following equation
\begin{numcases}{}\label{pz21}
{{\partial}\over{\partial t}}u_m+ H_m(t,x,u_m,D_{x}u_m, D^2_{x}u_m)=0,~(t,x) \in [0,T) \times \mathbb{R}^n,\nonumber\\
u_m(T,x)=h(x).
\end{numcases}

We then prove that the uniform convergence of $f_{m}$ to $f$, $Y_t^{t,x,m;v}$ to $Y_t^{t,x;v}$, $u_{m}$ to $u$ and $H_{m}$ to $H$ also holds in every compact subset of their own domains as $m\to\infty$, among which only the proof for the convergence of $f_{m}$ to $f$ is very different from Lemma 4.3 due to the different definitions of $f_m$ from $f_n$ and other convergence can be proved similarly according to Lemmas 4.4--4.6 in turn.

In fact, the uniform convergence of $f_{m}$ to $f$ in every compact subset of $[0,T]\times \mathbb{R}^n\times\mathbb{R}^1\times U$ is easy to see if we notice that for any given compact set $K\subset [0,T] \times \mathbb{R}^n \times\mathbb{R}^1\times U$, $f(t,x,y,v)$ is bounded by a positive integer $M_K$ for any $(t,x,y,v)\in K$ since the continuity of $f$. Hence, when $m\geq M_K$
\begin{eqnarray*}
&&\sup_{(t,x,y,v) \in K}|f_m(t,x,y,v)-f(t,x,y,v)|\nonumber\\
&=&\sup_{(t,x,y,v) \in K}|f(t,x,0,v)-\Pi_m\big(f(t,x,0,v)\big)|=0,
\end{eqnarray*}
which implies the uniform convergence of $f_{m}$ to $f$ in every compact subset of their domain as $m\to\infty$.

%
%
%

Finally, using Proposition 4.7 again we know that $u$ satisfies the limit equation of (\ref{pz21}), which together with the fact $u(T,x)=h(x)$ in view of the definition of value function shows that $u$ is still a solution of HJB equation (\ref{eq31}) even if $|f(t,x,0,v)|$ is not necessary to be uniformly bounded  for any $(t,x,v)\in[0,T]\times \mathbb{R}^n\times U$.
\end{proof}

\section{Example}

As mentioned in Introduction, Duffie and Epstein \cite{Duffie} presented the stochastic differential formulation of recursive utility which can be regarded as the solution of a BSDE. Based on this basic correspondence, we give an example to demonstrate the application of our study to utility.

We start from setting an financial market with two assets which can be traded continuously. One is the bond, a non-risky asset, whose price process $P^0_t$ is governed by the ordinary differential equation
\begin{equation}\label{exampleeq1}
P^0_t=1+\int_0^tr_sP^0_sds.
\end{equation}
The other asset is the stock, a risky asset, whose price process $P_t$  is modeled by the linear SDE
\begin{equation}\label{pz8}
P_t=p+\int_0^tP_sb_sds+\int_0^tP_s\sigma_sdB_s,\ \ \ \ {\rm where}\ p>0\ {\rm is}\ {\rm given}.
\end{equation}
In (\ref{exampleeq1}) and (\ref{pz8}), $r:[0,T]\to\mathbb{R}^1$ is the interest rate of the bond, $b:[0,T]\to\mathbb{R}^1$ is the appreciation rate of the stock and $\sigma:[0,T]\to\mathbb{R}^1$ is the volatility process, all of which are continuous functions.

A small agent whose actions cannot affect market prices may decide at time $t \in [0,T]$ what proportion of the wealth to invest in the stock. Denote the proportion and the wealth by  $\pi:\Omega\times[0,T]\to[-1,1]$ and $X:\Omega\times[0,T]\to\mathbb{R}^1$, respectively, then the equation which the increment of the wealth satisfies follows immediately:
\begin{numcases}{}\label{pz10}
dX_t = [r_t X_t +(b_t-r_t)\pi_tX_t-c_t]dt+ X_t\pi_t\sigma_tdB_t\nonumber\\
X_0=x,
\end{numcases}
where $c:\Omega\times[0,T]\to[a_1,a_2]$, $0\leq a_1<a_2$, is a restricted consumption decision and $x>0$ is the initial wealth of the investor. It is clear that (\ref{pz10}) acting as the state equation satisfies Conditions (H1) and (H2).

We assume that the stochastic differential utility preference of the investor is a continuous time Epstein-Zin utility as illustrated in (\ref{eq6}) and the utility satisfies the following BSDE:
\begin{numcases}{}\label{exampleeqv}
dV_t=-\frac{\delta}{1-{1\over\psi}}(1-\gamma)V_t\Big[ (\frac{c_t}{((1-\gamma)V_t)^{\frac{1}{1-\gamma}}})^{1-\frac{1}{\psi}} - 1 \Big]dt + Z_t dB_t\nonumber\\
V_T =h(X_T),
\end{numcases}
where $h:\mathbb{R}^1\to\mathbb{R}^1$ is a given Lipschitz continuous function.
The optimization objective of the investor is to maximize his/her utility as below:
\begin{equation*}
\max_{(\pi,c) \in \mathcal{U}}V_0,
\end{equation*}
where $$\mathcal{U}\triangleq \{(\pi,c)|\ (\pi,c):[0,T]\times\Omega\to[-1,1]\times[a_1,a_2]\ \text{is the}\ \text{$\{\mathscr{F}_t\}_{0\leq t\leq T}$-adapted process}\}$$ is the admissible control set.

Certainly, the aggregator of (\ref{exampleeqv}) does not satisfy the Lipschitz condition with respect to the utility and the consumption at all time, but we can find applications of our study in non-Lipschitz cases.  Notice that Proposition 3.2 in \cite{Kraft} provides four cases in which the aggregator is monotonic with respect to the utility. Taking into account the polynomial growth condition (H6) with respect to the utility we select two cases as follows for further consideration:
\begin{eqnarray*}
 &{\rm (i)}& \gamma>1  \text { and } \psi>1;\\
 &{\rm (ii)}& \gamma<1  \text { and } \psi<1.
\end{eqnarray*}
Then we can find suitable powers of utility such that the aggregator of (\ref{exampleeqv}) is continuous and monotonic but non-Lipschitz in $\mathbb{R}^1$ with respect to the utility in both cases. As for the continuity with respect to the consumption, if $a_1>0$, both cases are Lipschitz continuous obviously. In particular, if $a_1=0$,  only case (i) satisfies the continuous but not Lipschitz continuous condition with respect to the consumption.

Therefore, for all suitable non-Lipschitz situations which satisfy Conditions (H3)--(H6), we can use Theorem \ref{th2} to know that the value function of the investor is a viscosity solution of the following HJB equation:
\begin{numcases}{}
\max_{(\pi,c)\in[-1,1]\times[a_1,a_2]}\Big\{w_t(t,x)+[x(r_t+\pi(b_t-r_t))-c]w_x(t,x)+ \frac{1}{2}x^2\pi^2\sigma_t^2w_{xx}(t,x)\nonumber\\
\ \ \ \ \ \ \ \ \ \ \ \ \ \ \ \ \ \ \ \ \ \ +\frac{\delta}{1-{1\over\psi}}(1-\gamma)w(t,x)\Big[ (\frac{c}{((1-\gamma)w(t,x))^{\frac{1}{1-\gamma}}})^{1-\frac{1}{\psi}} - 1 \Big]\Big\}=0\nonumber\\
w(T,x)=h(x).\nonumber\\\vspace{2mm}
\end{numcases}
{\large\bf {Acknowledgements}}. The authors would like to thank Professor Shanjian Tang whose comments on the related topics initiated our motivation to do this work. Also we thank Dr. Fu Zhang for useful conversations.

\end{document}